%% file: paper.tex
\documentclass[twoside,11pt]{amsart}

\usepackage{amssymb}
\usepackage{amsmath}
\usepackage{epsfig,graphicx,psfrag}
\usepackage{multirow}
\usepackage[english]{babel}
\usepackage[dvips]{hyperref}
\hypersetup{
  pdftitle={},
  pdfauthor={},
  pdfsubject={},
  pdfkeywords={},
  colorlinks=false,
  breaklinks=true,
  bookmarksopen=true,
  bookmarksnumbered=true,
  pdfpagemode=UseOutlines,
  plainpages=false}

\newtheorem{theorem}{Theorem}[section]
\newtheorem{conclusion}[theorem]{Conclusion}

\newtheorem{corollary}[theorem]{Corollary}
\newtheorem{definition}[theorem]{Definition}
\newtheorem{example}[theorem]{Example}
\newtheorem{lemma}[theorem]{Lemma}
\newtheorem{proposition}[theorem]{Proposition}

\newenvironment{remark}{\vspace{4pt}\noindent\textbf{Remark.}}{\vspace{4pt}}

\newcommand{\IC}{\ensuremath{\mathbb{C}}}

\newcommand{\IN}{\ensuremath{\mathbb{N}}}
\newcommand{\IP}{\ensuremath{\mathbb{P}}}

\newcommand{\IZ}{\ensuremath{\mathbb{Z}}}

\newcommand{\cM}{\ensuremath{\mathcal{M}}}

\DeclareMathOperator{\Hom}{Hom}
\DeclareMathOperator{\id}{id}

\newcommand{\st}{\ensuremath{\;\mathrm{s.t.}\;}}

\numberwithin{equation}{section}
\begin{document}
\title[Order $6$ automorphisms]{Order $6$ non-symplectic automorphisms\\ of K3 surfaces.}
\author{Jimmy Dillies}
\address{University of Utah\\
Department of Mathematics\\
Salt Lake City, UT\\} 
\email{dillies@math.utah.edu}

\thanks{The author would like to thank Ron Donagi for his interest and 
support, Shigeyuki Kond\={o} for pointing out the discrepancy between 
section 
\ref{sec:pts} and the main theorem, in an earlier version of this paper and, Alessandra Sarti for many useful comments. }

% ... by reminding him of the $U(3)$ case illustrated in Remark \ref{rk:U3}

\begin{abstract}
We classify primitive non-symplectic automorphisms of order $6$ on K3 surfaces. We show how their study can be reduced to the study of non-symplectic automorphisms of order $3$ and to a local analysis of the fixed loci. In particular, we determine the possible fixed loci and show that when the Picard lattice is fixed, K3 surfaces come in mirror pairs. 
\end{abstract}



\maketitle

\setcounter{tocdepth}{1} {\scriptsize \tableofcontents}

%% Introduction

\section{Introduction}
\label{sec:intro}

An automorphism of a K3 surface is called non-symplectic when the induced action on the holomorphic $2$-form is non-trivial. The study of non-symplectic automorphisms was pioneered by Nikulin \cite{Ni81} who analyzed the case of involutions. 
Since then, these automorphisms have been extensively studied by several authors. Let us mention Vorontsov \cite{Vo83}, Kond\={o} \cite{Kon86,Kon92}, Xiao Gang \cite{Xi96}, Machida and Oguiso \cite{MO98}, Oguiso and Zhang \cite{OZ98,OZ00}, Zhang \cite{Zh07}, Artebani and Sarti \cite{AS08} and, Artebani, Sarti and Taki \cite{AST09}. \\
From these works, we now know that if a K3 surface admits a non-symplectic automorphism, then the surface is algebraic and the Euler totient function evaluated at the order of the automorphism is at most $66$. Moreover, non-symplectic automorphisms of prime order have been classified, a synthetic classification can be found in \cite{AST09}, and some authors have started to investigate the simulatenous existence of symplectic and non-symplectic automorphisms \cite{Fran09}.\\
One of the reasons behind the interest in non-symplectic involutions is the mirror construction of   Borcea \cite{Bor97} and Voisin \cite{Voi93}. They construct Calabi-Yau and an explicit mirror map using, as building blocks, K3 surfaces, with non-symplectic involutions, and elliptic curves. This construction can be extended to K3 surfaces with non-symplectic automorphisms of order 3, 4 and 6 \cite{Dil06}. 

In this paper, using the classification of non-symplectic automorphisms of order $3$ \cite{AS08}, we perform a local analysis of the action on the fixed locus and classify automorphisms of order $6$.

%%% Plan 

\
\section{Plan}
\label{sec:plan}

In Section \ref{sec:notation}, we define primitive non-symplectic automorphisms and fix the notation for the rest of the paper. In Section \ref{sec:res}, we give the final classification of all possible $\zeta$'s and their fixed locus. In Section \ref{sec:fix}, we show that the fixed locus of a primitive non-symplectic automorphism of order $6$ consists in a disjoint union of points, rational curves, and possibly one genus one curve. In Section \ref{sec:tri}, we separate our cases in $3$ families which are then analayzed in Sections \ref{sec:g1}, \ref{sec:ell} and \ref{sec:pts}. In Section \ref{sec:mod}, we discuss the moduli spaces of these families, and in the last Section, \ref{sec:pic}, we focus on the special case where $\zeta$ fixes the Picard lattice.

%% Notation

\section{Notation}
\label{sec:notation}

Let $X$ be a smooth projective K3 surface and $\zeta$ an automorphism of $X$. The induced action of $\zeta$ on $H(X,\Omega^2) \simeq \IC$ gives rise to a character $\chi$. An automorphism is called \emph{symplectic}, if it lies in the kernel of $\chi$, and \emph{non-symplectic} otherwise. If the order of $\zeta$ and  $\chi(\zeta)$ agree, then $\zeta$ is called \emph{primitive}. In the rest of the article, $\zeta$ will be a primitive non-symplectic automorphism of order $6$ acting on $X$. 

As suggested by Cartan \cite{Car57}, given a fixed point $P$ of $\zeta$, we can linearize the action around it. Since $\zeta$ is of order $6$ and primitive, the linearized action can be written as 
\[\begin{pmatrix}
\xi_6^k & 0\\
0 & \xi_6^{k'}
\end{pmatrix}\]
where $(k,k')\in \left\{(0,1);(2,5);(3,4) \right\}$ and $\xi_6$ is a primitive $6^\mathrm{th}$ root of unity. While the first case corresponds to $P$ lying on a fixed smooth curve, the last two options correspond to $P$ being isolated. We will use the standard notation and say that $P$ is of type $\frac{1}{6}(k,k')$.  
Since $\zeta$ is primitive, its iterates will also be non-symplectic. We will denote their fixed locus by $X^{[i]}=\{x\in X \st \zeta^i x=x\}$. The components of the $X^{[i]}$ will be described by the following variables:

\begin{itemize}
\item $p_{\frac{1}{n}(k,k')}$ : number of isolated fixed points of type $\frac{1}{n}(k,k')$ in $X^{\left[\frac{6}{n}\right]}$, for $n\in\{6,3,2\}$.
\item $l^{[i]}$ : number of rational curves in $X^{[i]}$.
\item $g^{[i]}$ : maximal genus among the curves in $X^{[i]}$.
\item $g_M=\mathrm{max}\{1,g^{[1]}\}$.
\end{itemize}

When referring to \cite{AS08}, we will use their notation, namely:
\begin{itemize}
\item $g$: highest genus of the curves in $X^{[2]}$. 
\item $n$: number of fixed points  in $X^{[2]}$ (all are of type $\frac{1}{3}(2,2)$).
\item $k$: total number of curves in $X^{[2]}$.
\end{itemize}

% Possible fixed loci

\section{Results}
\label{sec:res}

Our first result is a global description of the fixed locus of $\zeta$. 

\begin{theorem}
\label{thm:main.fix}
The fixed locus $X^{[1]}$ consists of one of the two following collections:  
\begin{enumerate}
\item a smooth genus $1$ curve and three isolated fixed points of type $\frac16 (2,5)$. 
\item  a disjoint union of smooth rational curves and points, $C_1 \sqcup \ldots C_l \sqcup P_1 \sqcup \ldots \sqcup P_{p_{\frac16 (3,4)} + p_{\frac16 (2,5)}}$, satisfying 
\begin{equation}
\label{eq:lhfr}
p_{\frac16 (3,4)} + 2p_{\frac16 (2,5)} - 6l^{[1]}  =6.
\end{equation}
\end{enumerate}
\end{theorem}

The proof of this Theorem follows from the lemmata and propositions of section \ref{sec:fix}. 

\begin{theorem}
 \label{thm:main.ext}
Let $X$ be a K3 surface and $\xi$ a non-symplectic automorphism of order $3$ of $X$.
There exists a K3 surface $X'$ and a primitive non-symplectic automorphism of order $6$  of $X'$, $\zeta$, such that
\begin{enumerate}
 \item $NS(X)=NS(X')$
 \item $X^\xi=X'^{\zeta^2}$
\end{enumerate}
When $X$ is $\zeta$-elliptic, i.e. when there is genus $1$ fibration on $X$ commuting with $\zeta$, then we can assume that $X'=X$ and that $\xi=\zeta^2$.\\
The case when the fixed locus of $\zeta$ contains an elliptic curve, is unique and studied in Proposition \ref{prop:g1}.
All other possible actions of $\zeta$ are described by their fixed locus in the Tables \ref{tab:results} and \ref{tab:results2}. 
\end{theorem}

Thus, given a N\'eron-Severi lattice L, there exists a K3 surface $X$ with $NS(X)=L$ which admits a non-symplectic automorphism of order $3$ if and only if there exists a surface $X'$ with $NS(X')=L$ which admits a non-symplectic automorphism of order $6$.\\

\begin{enumerate}
 \item Table \ref{tab:results} contains the fixed loci corresponding to the cases where $X$ is $\zeta$-elliptic, i.e. when there is genus $1$ fibration on $X$ commuting with $\zeta$. 
These cases are analyzed in Section \ref{sec:ell}.
The induced action, $\psi$, on the basis is either the identity or an involution. Each row begins by a description of $X^{[2]}$, then come the singular fibers of the fibration ($x$ is the number of fibers of type $X$). After that we give a description of $X^{[1]}$ when $\psi$ is the identity. Finally, if $\psi$ is an involution, we list the fibers above the two fixed points, and $X^{[1]}$. 

 \item Table \ref{tab:results2} contains the cases which are not $\zeta$-elliptic and are studied in Section \ref{sec:pts}.
\end{enumerate}

%\begin{remark}
%If $\xi$ is a $\xi$-elliptic non-symplectic automorphism of order $3$ on $X$, then  for $\xi$ to factor through an automorphism of order $6$ acting as involution on the base,  the number of $\xi$-fixed curves, $k$, has to divide $6$.
%\end{remark}

\input{table}

% Fixed locus

\section{Study of the fixed locus}
\label{sec:fix}

\begin{lemma}
\label{lem:hod}
The fixed locus of $\zeta$ consists in a disjoint union of smooth curves and points $$X^{[1]}=C_0 \sqcup \ldots C_m \sqcup P_1 \sqcup \ldots \sqcup P_{p_{\frac16 (3,4)} + p_{\frac16 (2,5)}}$$ with $ g(C_0) \geq 0 = g(C_1) = \ldots = g (C_m )$.
\end{lemma}

\begin{proof}
The first part of the statement follows from the Hodge Index Theorem. The argument is analogue to those found in \cite{Ni81,Voi93,Dil06,AS08}:\\
A disjoint union of smooth curves on a K3 surface can have at most one element with strictly positive self-intersection. By adjunction, that is a curve of genus at least 2.\\ 
If a curve has self-intersection $0$, then it is an elliptic curve and induces an elliptic fibration $\pi:X\rightarrow \IP^1$.   
Since the action is non-symplectic, it descends non-trivially to the base and fixes two points. The fixed locus of $\zeta$ is thus a component of the fibers above these two points. One of the fibers is the original fixed curve. The remaining curves of the fixed locus are either a smooth elliptic curve or a disjoint union of rational components of one of Kodaira's singular fibers. So either the fixed locus is as the one described in the statement, or it consists exactly in the disjoint union of two genus $1$ curves. However, if $X^{[1]}$ were to contain two genus $1$ curves, then so would $X^{[2]}$ and this option was ruled out in \cite{AS08}.
\end{proof}

\begin{lemma}
The components of $X^{[1]}$ satisfy 
\begin{equation}
\label{eq:lhf}
p_{\frac16 (3,4)} + 2p_{\frac16 (2,5)} - 6l^{[1]} + 6g_M =12.\end{equation}
\end{lemma}

\begin{proof}
The formula is simply the Lefschetz holomorphic formula, see \cite{TT75}, applied to $\zeta$. 
\end{proof}

We will use the classification, determined by Artebani and Sarti \cite{AS08}, of non-symplectic automorphism of order $3$ to find information on $X^{[2]}$, which in turn will yield us data on the nature of $X^{[1]}$. We first recapitulate the results of \cite{AS08}, which we will use later, and then relate the fixed loci $X^{[1]}$ and $X^{[2]}$.

\begin{theorem} 
\cite[Proposition 4.2]{AS08}
\label{thm:asell}
Let $\sigma$ be a non-symplectic automorphism of order $3$ acting on a K3 surface $X$. If the fixed locus of $\sigma$ contains two or more curves, then $X$ is isomorphic to an elliptic K3 surface whose Weierstrass equation is $$y^2=x^3+p_{12}(t)$$ and on which $\sigma$ acts as $(x,y,t)\mapsto(\zeta_3^2 x,y,t)$.
\end{theorem}

\begin{proposition}
\cite[Corollary 4.3]{AS08}
\label{prop:hell} 
Let $\sigma$ and $X$ be as in the statement of Theorem \ref{thm:asell}. If $X^{[1]}$ contains a curve $C$ of positive genus, then $C$ is a double section of the Weierstrass fibration, i.e. $C$ is hyperelliptic.
\end{proposition}

\begin{lemma}
\label{lem:623}
If $P\in X^{[1]}$ is of type $\frac16 (2,5)$, then it is also an isolated point in $X^{[2]}$.
If $P\in X^{[1]}$ is of type $\frac16 (3,4)$, then it lies on a smooth curve in $X^{[2]}$.
Moreover, one has the following inequalities $p_{\frac13 (2,2)} \geq p_{\frac16 (2,5)}$ and $l^{[2]} \geq l^{[1]}$.
\end{lemma}
 
\begin{proof}
The first two statements are obvious after one takes the square of the matrix giving the localized action of $\zeta$ at $P$. The inequalities ensue. 
\end{proof}

\begin{corollary}
\label{cor:ell} If $X^{[1]}$ contains at least two distinct curves or a curve and an isolated point of type $\frac16 (3,4)$, or more generally, if $X^{[2]}$ contains at least two distinct curves, then $X$ is isomorphic to an elliptic K3 surface whose Weierstrass equation is $$y^2=x^3+p_{12}(t).$$ Moreover, the action of $\zeta$ preserves the fibration, i.e. $\pi\circ\zeta\circ\pi^{(-1)}$ is well defined, and $\zeta^2$ acts as $(x,y,t)\mapsto(\zeta_3^2 x,y,t)$. In particular, the induced action on the base is at most of order $2$ and if this induced action is trivial, then $\zeta$ restricts to an action of order $6$ on each fiber.
\end{corollary}

\begin{proof}
Lemma \ref{lem:623} implies that $X^{[2]}$ contains at least two distinct curves. The first part of the statement follows thus directly from Theorem \ref{thm:asell}. \\
To show that $\zeta$ preserves the fibration, let us consider the action around a fixed point $P$ of $X^{[1]}$ lying on a fiber $F$ of $\pi$. Since the eigendirections of $\zeta$ and $\zeta^2$ agree and since $\zeta^2$ sends each fiber to itself, $\zeta$ preserves $F$: $\zeta(F)=F$. Given that $\pi$ is the map induced from the linear system $|F|$, $\zeta$ preserves the fibration. 
Finally, it is clear that the induced action of $\zeta$ on the base is at most of order $2$.
\end{proof}

\begin{definition}
\label{def:zeta}
If $X$, $\zeta$ are as in the statement of Corollary \ref{cor:ell}, we will say that $X$ is \emph{$\zeta$-elliptic}. The induced action on the basis will be denoted by $\psi=\pi\circ\zeta\circ\pi^{(-1)}$.
\end{definition}

\begin{lemma}
\label{lem:msection}
If $X$ is $\zeta$-elliptic than $X^{[i]}$ does not contain curves of strictly positive genus.
\end{lemma}

\begin{proof}
Assume that $X^{[1]}$ contains a curve $C_0$ which is not rational. Proposition \ref{prop:hell} tells us that $C_0$ is a double section of the fibration and therefore, the action induced on the base is trivial. By Corollary \ref{cor:ell}, $\zeta$ induces an automorphism of order $6$ on each fiber. Since $\zeta$ fixes at least two points per fiber, the points of intersection with $C_0$, it ought to be the identity: a contradiction.
\end{proof}

\begin{lemma}
\label{lem:g2}
The genus of $C_0$ is at most $1$.
\end{lemma}

\begin{proof}
If $g(C_0)\geq 2$, equation \ref{eq:lhf} implies that $l^{[1]}\geq 0$. Let us consider separately the cases when it is  strictly positive and when it is null.

\begin{enumerate}
\item If $l^{[1]}>0$, Corollary \ref{cor:ell} implies that $X$ is $\zeta$-elliptic and $C_0$ cannot have a strictly positive genus by Lemma \ref{lem:msection}.

\item If $l^{[1]}=0$ then equation \ref{eq:lhf} implies that there are no isolated fixed points and that the genus of $C_0$ is $2$. Since $\zeta$ acts as an involution on $X^{[2]}$, the fixed locus of $\zeta^2$ must contain an even number of fixed points and fixed rational curves. A close look at \cite[Table 1]{AS08} tells us that there are only two options:  $(p_{\frac13(2,2)},l^{[2]})$ is either $(2,0)$ or $(4,2)$. In the second case, $X$ is $\zeta$-elliptic, and that case is thus excluded by Lemma \ref{lem:msection}.
By \cite{AS08}, the case $(2,0)$ corresponds to $(X,\zeta^2)$ being isomorphic to a double cover of $\IP^2$ branched along a smooth sextic:
$$\begin{cases}
X: Y^2=F_6(X_0,X_1)+F_3(X_0,X_1)X_2^3+fX_2^6 \\ 
\zeta^2: (X_0:\ldots:X_2,Y)\mapsto(X_0:X_1:\xi^2 X_2,Y)
\end{cases}.$$
The fixed locus of $\zeta^2$ consists of the genus $2$ curve $C_0:Y^2=F_6(X_0,X_1)$ and of the two points over the origin $P_{1,2}: Y^2=fX_2^6$. 
Since $\zeta$ fixes no isolated points, $\zeta$ must permute the two "sheets" of $X$. However, this involution acts non-trivially on $C_0$ and it only fixes the $6$ points of intersection of the curve with the plane $Y=0$, which is a contradiction.
\end{enumerate}
\end{proof}

%% Triage

\section{Triage}
\label{sec:tri}

In this section we show that except for $8$ possible configurations of $(X^{[1]},X^{[2]})$, all other cases are $\zeta$-elliptic, in the sense of Definition \ref{def:zeta}. 

\begin{proposition}
If the fixed locus of $\zeta$ contains a rational curve then $X$ is $\zeta$-elliptic.
\end{proposition}

\begin{proof}
When the fixed locus contains at least one curve, formula \ref{eq:lhf} reduces to $p_{\frac16 (3,4)}+2p_{\frac16 (2,5)}-6l^{[1]}=6$. If $p_{\frac16 (3,4)}$ is strictly positive, then Corollary \ref{cor:ell} implies that $X$ is elliptic. Otherwise, $p_{(2,5)}\geq 6$ and thus $n$ is an odd number at greater or equal to $6$. From \cite[Table 2]{AS08} one can see that all cases where $X^{[2]}$ contains a rational curve and where $n$ is an odd number larger than $6$ are $\zeta$-elliptic.
\end{proof}

\begin{proposition}
\label{prop:nell}
If the fixed locus of $\zeta$ contains only points, then $X$ is elliptic except for the possible following $8$ cases:
$\left( p_{\frac16 (3,4)}, p_{\frac16 (2,5)}; n, k, g \right)$ 
$\in$ $\left\{ (6,0;0,1,4),\right.$  
$(6,0;2,1,2),$ $\left.(4,1;1,1,3),\right.$ $(4,1;3,1,1),$ $(2,2;2,1,2),$ $(2,2;4,1,0),$ $(0,3;3,0,\emptyset),$ $\left.(0,3;3,1,1) \right\}.$
\end{proposition}

\begin{proof}
From the holomorphic Lefschetz formula \ref{eq:lhf}, the points in $X^{[1]}$ satisfy $p_{\frac16 (3,4)}+ 2 p_{\frac16 (2,5)}=6$, i.e. $(p_{\frac16 (3,4)}, p_{\frac16 (2,5)}) \in \left\{ (6,0), (4,1), (2,2), (0,3)\right\}$.
On the other hand, \cite[Table 2]{AS08} tells us that there are only $6$ possibilities for $(n,k,g)$ with $X$ non-elliptic:
$$\left\{ (0,1,4), (1,1,3), (2,1,2), (3,0,\emptyset), (3,1,1), (4,1,0) \right\}.$$  A priori, there are thus $24$ combinations of $X^{[1]}$ and $X^{[2]}$ to analyze.\\
In lemma \ref{lem:623} we saw that $n\geq p_{\frac16 (2,5)}$. This excludes $6$ possibilities.  
Since $\zeta$ acts as an involution on $X^{[2]}$, $n$ is congruent to $p_{\frac16 (2,5)}$ modulo $2$. This excludes an additional $8$ combinations. Since points of type $\frac16 (3,4)$ lie on fixed curve in $X^{[2]}$, the combination $(4,1;3,0,\emptyset)$ is impossible. Finally, an involution on a rational curve has only $2$ fixed points, hereby excluding the case $(6,0;4,1,0)$.
The listed combinations are the $24-6-8-1-1=8$ remaining possibilities.
\end{proof}

\begin{conclusion}
The fixed locus of $\zeta$ consists in a disjoint union of smooth rational curves and points and possibly one elliptic curve. If there is an elliptic curve then the action is essentially unique, and will be described in \ref{prop:g1}. Section \ref{sec:pts} will analyze the $8$ cases of Proposition \ref{prop:nell}, while all other cases, which correspond to $X$ elliptic, will be discussed in section \ref{sec:ell}. 
\end{conclusion}

%% Case X[1]=Ell & 3 ptsb

\section{The case where the fixed locus contains a genus 1 curve}
\label{sec:g1}

\begin{lemma}
If $g(C_0)=1$ then $p_{\frac16 (3,4)}=l^{[1]}=0$, $p_{\frac16(2,5)}=3$ and $l^{[2]}=0$, $p_{\frac13(2,2)}=3$.
\end{lemma}

\begin{proof}
If $l^{[1]}$ or $p_{\frac16 (3,4)}$ were to be strictly positive, Corollary \ref{cor:ell} would  imply that $X$ is $\zeta$-elliptic contradicting Lemma \ref{lem:msection}. Formula \ref{eq:lhf} gives us the value of $p_{\frac16 (2,5)}$. \\
Similarly, the case $l^{[2]}>0$ is excluded as we would reach a similar contradiction. Finally, the value for $p_{\frac13 (2,2)}=3$ can be found in \cite[Table 1]{AS08}.
\end{proof}

\begin{proposition}
\label{prop:g1}
If the fixed locus contains an elliptic curve, then the action is defined uniquely, i.e. the fixed loci of $\zeta$, $\zeta^2$ and $\zeta^3$ are determined uniquely: $X^{[1]}=X^{[2]}$ consists in the elliptic curve and three isolated points, while the fixed locus of $\zeta^3$ contains the elliptic curve and another elliptic curve containing the above $3$ points.
\end{proposition}

\begin{proof}
Consider the elliptic fibration given by the linear system $|C_0|$. Since $C_0$ is in the fixed locus, the induced action on the base is of order $6$, i.e. it is a cyclic action with two fixed points: the image of $C_0$ and some additional point $Q$. Since the Euler characteristic of a K3 surface is $24$, the Euler characteristic of the fiber above $Q$ is a multiple of $6$. From Kodaira's classification of the possible singular fibers, the fiber above $Q$ is of the type $I_{6N}$ or $I^*_{6N}$. However, as will follow from section \ref{subsec:loc}, only in the case $I_0$ does $\zeta$ not fix any rational curves. The fiber above $Q$ is thus smooth and the fixed loci of $\zeta$ and its powers are readily found.
\end{proof}

\begin{remark}
An example of a K3 with a primitive non-symplectic automorphism of order $6$ fixing an elliptic curve is given by the surface $y^2=x^3+(t^6-1)^2$, where the action is $\zeta:(x,y,t)\mapsto (x,y,\xi_6 t)$. The volume form $\omega=\frac{dx\wedge dt}{dy}$ gets mapped to $\zeta^* \omega=\xi_6 \omega$.
\end{remark}

%% 2 Elliptic case 

\section{Elliptic case}
\label{sec:ell}

In this section we consider $X$ to be $\zeta$-elliptic. The induced automorphism, $\psi$, on $\IP^1$, is either trivial or an involution. The two cases are analyzed respectively in sections \ref{subsec:triv} and \ref{subsec:nontriv}. Our discussion begins in section \ref{subsec:loc} where we analyze how $\zeta$ acts on the fibers of $\pi$.

\subsection{Local analysis}
\label{subsec:loc}

Let $X$ be a K3 surface. The \emph{Gram graph} of $X$ is the incidence graph of the effective smooth rational curves on $X$. 
E.g., when the Picard lattice of $X$ is isomorphic to $U\oplus E_8^2$ of $S_X$, then the Gram graph is as in figure \ref{fig:gram}. Let $D$ be an effective divisor on $X$ and $\zeta$ an automorphism of $X$, we call $D$ \emph{stable} if $\zeta(D)=D$, and we say that $D$ is \emph{fixed} if $\zeta \vert_D=\id$. 

\begin{figure}[htp]
\centering
\includegraphics[height=0.6cm]{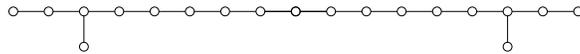}
\caption{Gram graph of $U\oplus E_8^2$.}\label{fig:gram}
\end{figure}

\begin{lemma} \label{lem:tree}
Consider a tree of rational curves on a surface $X$ which are stable componentwise under the action of a primitive non-symplectic automorphism of order $6$. Then, the points of intersection of the rational curves are fixed and the action at one fixed point determines the action on the whole tree. 
\end{lemma}

\begin{proof}
The key in this proof is to realize that the action of the automorphism on a given rational component and the action on a fixed point of this curve determine each other completely. Recall that an action of $\IC$ will be of the form $z\mapsto \lambda z$, $\lambda \in \IC^*$ under suitable coordinates. Now, $\lambda$ is nothing but the eigenvalue associated to the fixed point of coordinate $0$, or the inverse of the eigenvalue associated to the fixex point at infinity. Conversely, if one knows one eigenvalue of the automorphism localized at a point, then one knows the full action at that point. First, the eigendirections correspond to the components of the tree passing through the point Secund, since the three types of points, $\frac16(3,4)$, $\frac16(2,5)$ and $\frac16(1,0)$, all have distinct eigenvalues it is clear to which eigenvalue corresponds each direction.
\end{proof}

\begin{remark}
It follows from the proof of the previous lemma that if we look at the types of points of intersection on a chain of smooth rational curves, these will embed in the following periodic sequence:
$$ \ldots, \frac16(2,5), \frac16(3,4), \frac16(3,4), \frac16(2,5), \frac16(1,0), \frac16(1,0), \ldots$$
\end{remark}

\begin{example}
\label{ex:E6}
Consider a type $IV^*$ configuration of rational curves which is stable under the action of $\zeta$, a non-symplectic automorphism of order $6$. Moreover, assume that it contains on one of the weight $1$ curves, $L$, a point $P$ of type $\frac16 (3,4)$, such that the eigendirection corresponding to the eigenvalue $-1$ is transversal to the $L$. Using Lemma \ref{lem:tree} we can determine the action on the entire configuration. This action is illustrated in Figure \ref{fig:E6b}.
\end{example}

\begin{figure}[htp]
\centering
\psfrag{P}{$P$}
\psfrag{A}{$\frac16 (3,4)$}
\psfrag{B}{$\frac16 (2,5)$}
\psfrag{C}{$\frac16 (1,0)$}
\includegraphics[height=2cm]{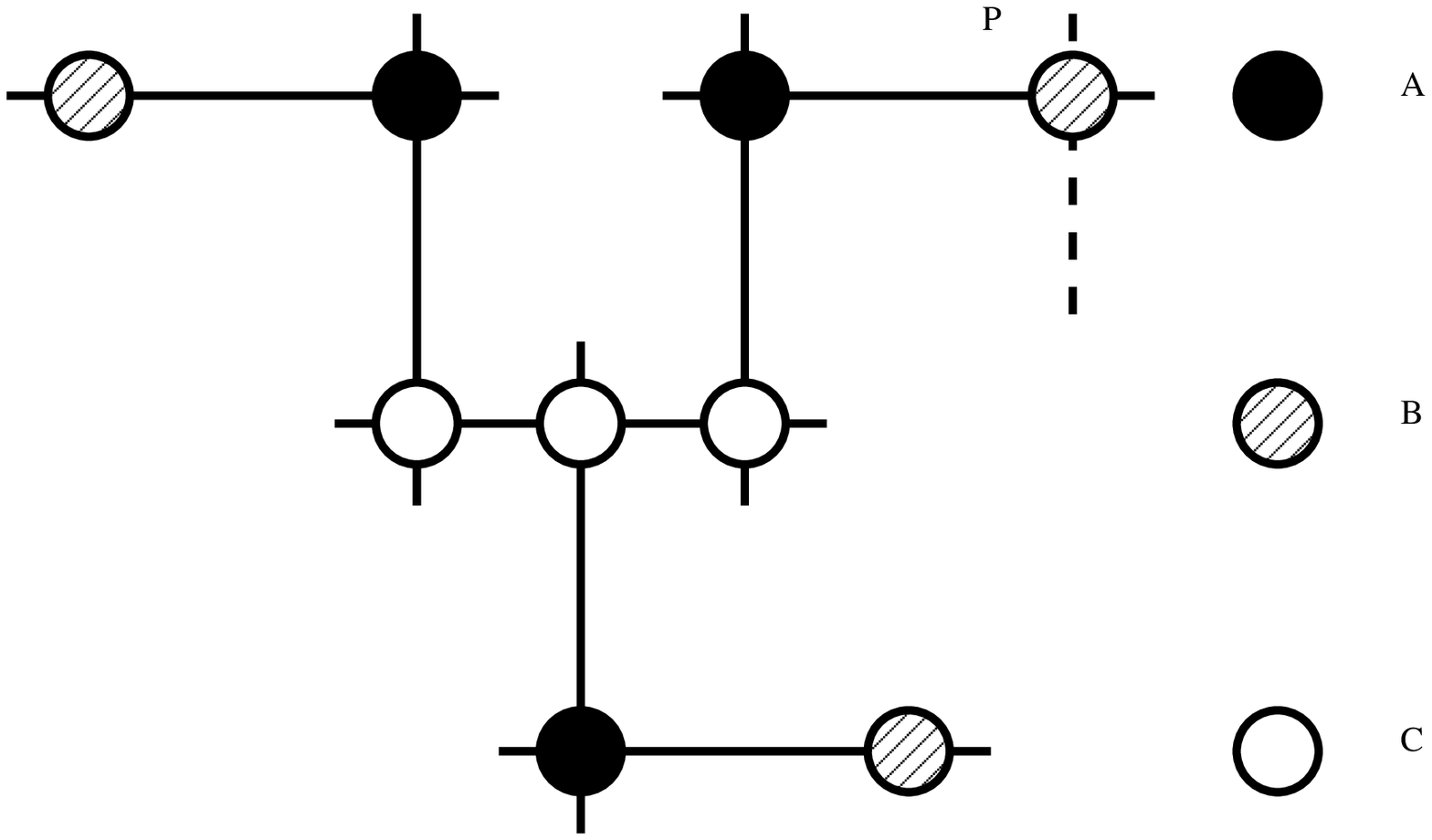}\\
\caption{Study of the action on a type $IV^*$ configuration.}\label{fig:E6b}
\end{figure}

From now on, we will focus only on fibers of type $I_0$, $II$, $II^*$, $IV$ and $IV^*$. We will denote by $ii$ the number of type $II$ fibers, $ii^*$ the number of type $II^*$ fibers, etc. We focus our attention on these fibers because of the following theorem: 

\begin{proposition}
\cite[Proposition 4.2]{AS08})
\label{prop:AS42}
Let $X$ be $\zeta$-elliptic, then the numbers $(n,k)$ determine uniquely $ii$, $ii^*$, $iv$ and $iv^*$.
More precisely, $\pi$ has
\begin{enumerate}
 \item $n$ type $IV$ fibers if $k=2$.
 \item $n-3$ type $IV$ fibers and $1$ type $IV^*$ fiber if $k=3$.
 \item $n-4$ type $IV$ fibers and $1$ type $II^*$ fiber if $k=4$.
 \item $n-7$ type $IV$ fibers, $1$ type $IV^*$ fiber and  and $1$ type $II^*$ fiber if $k=5$.
 \item $n-8$ type $IV$ and $2$ type $II^*$ fibers if $k=6$.
\end{enumerate}
\end{proposition}

\begin{lemma}
\label{lem:loc}
Let $X$ be $\zeta$-elliptic, $\psi$ trivial, and assume that $X$ has a fiber of type $II$, $IV$, $II^*$ or $IV^*$. When restricted to those fibers, $\zeta$  
\begin{enumerate}
 \item fixes $1$ point of type $\frac16 (3,4)$, namely the cuspidal point of the fiber. (Fiber of type $II$)
 \item fixes $3$ points of type $\frac16 (3,4)$, $4$ points of type $\frac16 (2,5)$ and $1$ rational curve. (Fiber of type $II^*$)
 \item fixes $1$ point of type $\frac16 (2,5)$, namely the common intersection point. (Fiber of type $IV$)
 \item fixes $2$ points of type $\frac16 (3,4)$ and $1$ point of type $\frac16 (2,5)$. (Fiber of type $IV^*$)
\end{enumerate}
Moreover, $\pi$ has also a section fixed by $\zeta$ and it is the only part of $X^{[1]}$ not completely included in the fibers.
\end{lemma}

\begin{proof}

Since $\psi$ is trivial, the fibers are preserved by $\zeta$. Thus either $\zeta$ fixes the zero section $\sigma_0$, or there is another section $\sigma_1$ and $\zeta$ permutes the two. Assume $\sigma_0$ is not fixed. Pick a smooth fiber $F$ of $\pi$. The automorphism $\zeta^2$ is of order $3$ on $F$ and fixes the $2$ points of intersection with the two sections $\sigma_0$ and $\sigma_1$. Therefore, there is a third fixed point. Since $\zeta$ permutes the first two, it fixes the third one. Since $\psi$ is trivial, this point is of type $\frac16 (1,0)$ and there is a fixed section passing through that point.

Let us describe the action on the fibers explicitly.

\begin{enumerate}

\item $(II)$ The point of intersection with the fixed section is not the node, as the section intersects the fiber with multiplicity $1$, and is of type $\frac16(0,1)$. On the other hand, the other fixed point, which ought to be the node, is of type $\frac16(3,4)$. 

\item $(II^*)$ Since the Gram graph of this fiber has no non-trivial $\IZ/2\IZ$ automorphism, the curve of weight $6$ is fixed. The remaining fixed points can be found using Lemma \ref{lem:tree}.

\item $(IV)$ The section of the Weierstrass fibration is fixed and intersects the fiber at the curve of weight $1$. Using Lemma \ref{lem:tree} we see that there is a unique possible action, namely the one permuting the two other branches. 

\item $(IV^*)$ The action on the fiber follows from lemma \ref{lem:tree} and is described in figure \ref{fig:E6}. The black dot corresponds to a point of type $\frac16(2,5)$ and the two white dots to points of type $\frac16(3,4)$.

\begin{figure}[htp]
\centering
\psfrag{A}{$\frac16 (3,4)$}
\psfrag{B}{$\frac16 (2,5)$}
\psfrag{C}{$\frac16 (1,0)$}
\includegraphics[height=2cm]{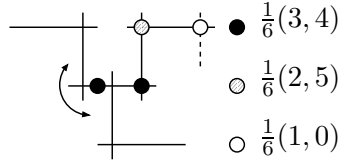}\\
\caption{Action on a type $IV^*$ fiber when $\psi$ is trivial.}\label{fig:E6}
\end{figure}

\end{enumerate}
\end{proof}

\begin{lemma}
\label{lem:loc2}
Let $X$ be $\zeta$-elliptic and assume $\psi$ is an involution.  Let $F$ be a fiber preserved by $\zeta$, i.e. $\zeta(F)=F$. $F$ is of type $I_0$, $IV$ or $IV^*$. Moreover, $\zeta$ 
\begin{enumerate}
 \item fixes $3$ points of type $\frac16 (3,4)$, when $F$ is smooth. 
 \item fixes $1$ point of type $\frac16 (3,4)$ and $1$ point of type $\frac16 (2,5)$; when $F$ is of type $IV$. This case is depicted in Example \ref{ex:E6}.
 \item fixes $3$ points of type $\frac16 (3,4)$, $3$ points of type $\frac16 (2,5)$ and $1$ rational curve, when $F$ is of type $IV^*$.
\end{enumerate}
Moreover, every component of $X^{[1]}$ lies in one of those fibers.
\end{lemma}

\begin{proof}
Without loss of generality, we can assume that $\psi$ is of the form $[x_0:x_1]\mapsto [-x_0:x_1]$, or $t\mapsto -t$. Since, the Weierstrass equation $y^2=x^3+p_{12}(t)$ is invariant under $\psi$, this implies that that the roots of $p_{12}$ are double at $0$ and $\infty$. The fibers which correspond to double roots are those of type $I_0$, $IV$ and $IV^*$. Alternatively, one can perform a local analysis on the fibers, and see that these are the only possibilities. This analysis will also give us the exact nature of the fixed locus on each fiber.
Take a fixed point $P$ at the intersection of the section of $\pi$ and a fiber $F$. Since $\psi$ is an involution, the eigenvalue corresponding to the direction of the section is $-1$. Using Lemma \ref{lem:tree} we can describe the local action in each case:
\begin{enumerate}
 \item $(I_0)$ There are $3$ fixed points of type $\frac16 (3,4)$.
 \item $(IV)$ There is $1$ fixed point of type $\frac16 (3,4)$, and one of type $\frac16 (2,5)$. 
 \item $(IV^*)$ There are $3$ points of type $\frac16 (3,4)$, $3$ points of type $\frac16 (2,5)$ and $1$ rational curve.
\end{enumerate}

\end{proof}

\subsection{Induced action on the base is trivial}
\label{subsec:triv}

\begin{lemma}
Let $X$ be $\zeta$-elliptic, $\psi$ trivial, then $X^{[2]}$ determines $X^{[1]}$. The possibilities are listed in Table \ref{tab:results}.
\end{lemma}

\begin{proof}
From Proposition \ref{prop:AS42}, we know that $(n,k)$ determines the types of fibers of $\pi$. Since, Lemma \ref{lem:loc} tells us that the action of $\zeta$ on each fiber is unique, it follows that $X^{[1]}$ is completely determined by $X^{[2]}$.
\end{proof}

\begin{remark}
Unfortunately, the converse is not true: the simple combinatorial data describing $X^{[1]}$ does not determine uniquely $X$ or $X[2]$. See examples $4$ and $10$ in Table \ref{tab:results}.
\end{remark}

Finally, we will show the existence of all the examples in Table \ref{tab:results}.

\begin{lemma}
Let $X$ be a K3 surface and $\tau$ a non-symplectic automorphism of order $3$ which is $\tau$-elliptic. Then there exists a primitive non-symplectic automorphism $\zeta$ of $X$ such that $\tau=\zeta^2$ and $X$ is $\zeta$-elliptic.
\end{lemma}

\begin{proof}
We know from \ref{thm:asell} that $X$ is of the form $y^2=x^3+p_{12}(t)$ and $\tau$ acts as $(x,y,t)\mapsto(\xi_3 x,y,t)$. 
It is easy to see that $\sigma:(x,y,t)\mapsto(\xi_3 x,-y,t)$ acts on $X$ and has the required properties.
\end{proof}

\subsection{Induced action on the base is an involution}
\label{subsec:nontriv}

\begin{lemma}
The fixed locus $X^{[1]}$ is contained in $2$ fibers of $\pi$. 
\end{lemma}

\begin{proof}
Since $\psi$ is an involution, it has two fixed points on $\IP^1$, say $0$ and $\infty$.  Since all the fibers not above these points are permuted, $X^{[1]}$ is a subset of the fibers $F_0$ and $F_\infty$ ($F_i=\pi^{(-1)}(i)$).
\end{proof}

\begin{lemma}
Let $X$ be a K3 surface and $\tau$ a non-symplectic automorphism of order $3$ which is $\tau$-elliptic. Call $\pi$ the associated fibration. Assume that the multiset $X_\pi$ of singular fibers of $\pi$ can be decomposed $F \sqcup M$ with $F$ a multiset of cardinality $2$ whose elements come from $\{ I_0, IV, IV^* \}$ and where each element of $M$ has even multiplicity. 
Then there exists a pair $(X',\tau')$ consisting of a K3 surface and a non-symplectic automorphism of order $3$ such that $X'$ is $\tau'$-elliptic, $X_\pi = X'_{\pi'}$, and $X^\tau=X'^{\tau'}$. Moreover, $\tau'$ factors as $\tau'=\zeta^2$ where $\zeta$ is a primitive non-symplectic automorphism of order $6$ commuting with $\pi'$.
\end{lemma}

\begin{proof}
This follows from the local analysis in Lemma \ref{lem:loc2}, or from the fact that the only singular fibers corresponding to double roots of $p_{12}(t)$ are those of $I_0$, $IV$, and $IV^*$. 
\end{proof}

Since the action on $F_0$ and $F_\infty$ is determined by Lemma \ref{lem:loc2}, we simply list all possibilities in Table \ref{tab:results}.

%%% 3 Non-elliptic case

\section{Non-elliptic case}
\label{sec:pts}

In this section, we analyze the $8$ remaining cases of Proposition \ref{prop:nell}. 
Given, $(X,\zeta)$, we start by giving a projective model isomorphic to $(X,\zeta^2=\sigma)$. These models come from \cite[Propositions 4.7, 4.9 and 4.11]{AS08}. \\

We will denote by $F_d$ or $G_d(x_0^{m_0},\ldots,x_n^{m_n})$ a general homogenous polynomial in the $x_i^{m_i}$, of \emph{total} degree $d$. Recall that the numbers between parentheses are 
$\left( p_{\frac16 (3,4)}, p_{\frac16 (2,5)}; n, k, g \right)$. \\

When $(n,k,g)=(2,1,2)$ we show that there is a canonical non-symplectic involution commuting with our original non-symplectic automorphism of order $3$ and, in some special cases, one might be able to find another factorization by a non-symplectic automorphism of order $6$. In the other cases, we will always have to restrict ourselves to a subfamily of the moduli of K3 surfaces with non-symplectic involutions of order $3$. In each other projective case, given a triple $(n,k,g)$, there exists a K3 surface endowed with a non-symplectic automorphism $\zeta$ of order $6$ such that the fixed locus of $\zeta^2$ is of type $(n,k,g)$, but the surface is in general different from the original one. \\

We begin by the showing that the $(2,1,2)$ case always factors.

\begin{enumerate} 
 \item $(6,0;2,1,2)$
	$$X: \left. \begin{array}{rcl} 
	y^2 &=  &F_6(x_0,x_1)+F_3(x_0,x_1)x_2^3+b x_2^6
	\end{array}\right.$$
	$$\sigma:(x_0,x_1,x_2)\mapsto(x_0,x_1,\xi_3 x_2)$$

In this case, $\sigma$ always factors through the primitive non-symplectic automorphism of order $6$ :
	$$\zeta:(x_0,x_1,x_2;y)\mapsto(x_0,x_1,\xi_3^2 x_2;-y)$$

\end{enumerate}

Let us now work on the other cases where a change of surface is necessary. 
The first case is worked out in detail. The other cases are similar to treat. Note that the last configuration is not possible. It is also discussed in detail hereunder.

\begin{enumerate}

 \item $(6,0;0,1,4)$
	$$X:\left\{ \begin{array}{rcl}
	F_2(x_0,x_1,x_2,x_3) &= &0\\
	F_3(x_0,x_1,x_2,x_3) + bx_4^3 &=&0
	         \end{array}\right.$$ 
	$$\sigma:(x_0,x_1,x_2,x_3,x_4)\mapsto (x_0,x_1,x_2,x_3,\xi_3 x_4)$$

The fixed locus $X^{[1]}$ consists of $6$ points of type $\frac16 (3,4)$, which must lie on the only fixed curve in $X^{[2]}$, namely the curve of genus $4$, intersection of $X$ and $\left\{(x_0,\ldots,x_3,0)\right\}$. Note that in general, $\sigma$ does not factor through an automorphism of order $6$. Indeed, all K3 surfaces of genus 3 come as complete intersection in $\IP^3$ but not all admit a 2:1 map to an elliptic curve. To have such a map, we need
extra symmetry: let $F_2$ and $F_3$ be polynomials in $x_0^2$ which admit as involution the change of sign on the first coordinates. This will fix exactly six points. Thus, we can assume take $\zeta$ is of the form:
 
	$$X:\left\{ \begin{array}{rcl}
	F_2(x_0^2,x_1,x_2,x_3) &= &0\\
	F_3 (x_0^2,x_1,x_2,x_3) + bx_4^3 &=&0
	\end{array}\right.$$ 
	$$\zeta:(x_0,x_1,x_2,x_3,x_4)\mapsto (-x_0,x_1,x_2,x_3,\xi_3^2 x_4)$$

\item $(2,2;2,1,2)$
	$$X:\left. \begin{array}{rcl} 
	y^2 &=  &F_6(x_0,x_1)+F_3(x_0,x_1)x_2^3+b x_2^6
	\end{array}\right.$$
	$$\sigma:(x_0,x_1,x_2)\mapsto(x_0,x_1,\xi_3 x_2)$$

	$$X:\left. \begin{array}{rcl} 
	y^2 &=  &F_6(x_0^2,x_1)+F_3(x_0^2,x_1)x_2^3+b x_2^6
	\end{array}\right.$$
	$$\sigma:(x_0,x_1,x_2)\mapsto(-x_0,x_1,\xi_3^2 x_2)$$

 \item $(4,1;1,1,3)$
	$$X:\left. \begin{array}{rcl}
	F_4(x_0,x_1,x_2)+F_1(x_0,x_1,x_2)x_3^3 &= &0
	         \end{array}\right.$$ 
	$$\sigma:(x_0,x_1,x_2,x_3)\mapsto (x_0,x_1,x_2,\xi_3 x_3)$$
	$$X:\left. \begin{array}{rcl}
	F_4(x_0^2,x_1,x_2)+F_1(x_0,x_1,x_2)x_3^3 &= &0
	         \end{array}\right.$$ 
	$$\tau:(x_0,x_1,x_2,x_3)\mapsto (-x_0,x_1,x_2,\xi_3^2 x_3)$$

 \item $(4,1;3,1,1)$
	$$X:\left\{ \begin{array}{rcl}
	x_3F_1(x_0,x_1,x_2)+x_4G_1(x_0,x_1,x_2) &= &0\\
	F_3(x_0,x_1,x_2)+G_3(x_3,x_4) &= &0
	         \end{array}\right.$$
	$$ \sigma:(x_0,x_1,x_2,x_3,x_4)\mapsto (x_0,x_1,x_2,\xi_3 x_3,\xi_3 x_4)$$
	$$X:\left\{ \begin{array}{rcl}
	x_4G_1(x_0,x_1,x_2) &= &0\\
	F_3(x_0^2,x_1,x_2)+G_3(x_3^2,x_4) &= &0
	         \end{array}\right.$$
	$$ \tau:(x_0,x_1,x_2,x_3,x_4)\mapsto (-x_0,x_1,x_2,-\xi_3^2 x_3,\xi_3^2 x_4)$$

 \item $(2,2;4,1,0)$
	$$X:\left. \begin{array}{rcl} 
	F_4(x_0,x_1)+F_3(x_2,x_3)F_1(x_0,x_1) &= &0
	\end{array}\right.$$
	$$\sigma:(x_0,x_1,x_2,x_3)\mapsto(x_0,x_1,\xi_3 x_2, \xi_3 x_3)$$
 	$$X:\left. \begin{array}{rcl} 
	F_4(x_0^2,x_1)+F_3(x_2^2,x_3)F_1(x_1) &= &0
	\end{array}\right.$$
	$$\tau:(x_0,x_1,x_2,x_3)\mapsto(-x_0,x_1,-\xi_3^2 x_2, \xi_3^2 x_3)$$

 \item $(0,3;3,0,\emptyset)$
	$$\left\{ \begin{array}{rcl}
	F_2(x_0,x_1)+bx_2 x_3 + c x_2 x_4 &= &0\\
	f_3(x_0,x_1) + d x_2^3 + G_3(x_3,x_4) + x_2 F_1(x_0,x_1) G_1(x_3,x_4) &= &0
	\end{array}\right.$$
	$$ \sigma:(x_0,x_1,x_2,x_3,x_4)\mapsto (x_0,x_1,\xi_3^2 x_2,\xi_3 x_3,\xi_3 x_4)$$
	$$\left\{ \begin{array}{rcl}
	F_2(x_0^2,x_1)+bx_2 x_3 + c x_2 x_4 &= &0\\
	F_3(x_0^2,x_1) + d x_2^3 + G_3(x_3,x_4) + x_2 G_1(x_1) G_1(x_3,x_4) &= &0
	\end{array}\right.$$
	$$ \tau:(x_0,x_1,x_2,x_3,x_4)\mapsto (-x_0,x_1,\xi_3 x_2,\xi_3^2 x_3,\xi_3^2 x_4)$$

 \item $(0,3;3,1,1)$ is impossible.
	$$X:\left\{ \begin{array}{rcl}
	x_3F_1(x_0,x_1,x_2)+x_4G_1(x_0,x_1,x_2) &= &0\\
	F_3(x_0,x_1,x_2)+G_3(x_3,x_4) &= &0
	         \end{array}\right.$$
	$$ \sigma:(x_0,x_1,x_2,x_3,x_4)\mapsto (x_0,x_1,x_2,\xi_3 x_3,\xi_3 x_4)$$

One way to show that this configuration is impossible is to use the above projective model. However, we will show it using the fact that the elliptic curve in $X^{[2]}$ defines a fibration of $X$. Since the Picard lattice of $X$ is $U(3)\oplus A_2^3$, see \cite[{Table 2}]{AS08}, the fibration has a triple section $\sigma$. The image of $\sigma$ intersects the elliptic curve in $X^{[2]}$ with multiplicity $3$. Since $\zeta$ acts as a fixed point free involution on this elliptic curve, the intersection $\sigma$ and the elliptic curve contains two points; this  contradicts the oddness of their intersection number.

\end{enumerate}

%% 4 Moduli & Rejults

\section{Moduli}
\label{sec:mod}

As a direct consequence to Theorem \ref{thm:main.ext} we have the following corollary:

\begin{corollary}
The moduli space of K3 surfaces admitting a primitive non-symplectic automorphism of order $6$ acting fiberwise on an elliptic fibration agrees with the moduli of K3 surfaces admitting a non-symplectic automorphism of order $3$ acting fiberwise on an elliptic fibration.
\end{corollary}

The latter has been analyzed and can be described as follows:

\begin{theorem} \cite[Propositions 5.1 and 5.6]{AS08} \label{thm:mod}
Let $\cM_{n,k}$ be the moduli space of K3 surfaces with an order $3$ non-symplectic automorphism fixing $n$ points and $k$ curves. All $\cM_{n,k}$ are irreducible and the total moduli of K3 surfaces with  an order $3$ non-symplectic automorphism has three irreducible components, namely the closures of $\cM_{0,1}$, $\cM_{0,2}$ and $\cM_{3,0}$. Moreover the $\cM_{n,k}$ with $k>1$ lie in the closure of $\cM_{0,2}$ and the $\cM_{j,1}$ lie in the closure of $\cM_{0,1}$.
\end{theorem}

While for automorphisms of order $3$, the fixed locus determines the automorphism, the same is not true for order $6$ non-symplectic automorphism. Neither $X^{[2]}$ nor $X^{[1]}$ completely determine $X$ or the automorphism. 
If we wish to look at the moduli of K3 surfaces together with an automorphism of order $6$, the structure is slightly more intricate.

\begin{definition}
Let $\cM_{n,k}^6$ be the moduli space of K3 surfaces with an order $6$ non-symplectic automorphism $\zeta$ such that $\zeta^2$ fixes $n$ points and $k$ curves.
\end{definition}

While the underlying space of $\cM_{n,k}^6$ is a subvariety of $\cM_{n,k}$, at some points, $\cM_{n,k}^6$ has a stacky structure coming from the extra ways of extending the order $3$ automorphism.

\begin{remark}

\begin{enumerate}
\item $\cM_{n,k}^\mathrm{id}$ are the moduli spaces of $\zeta$-elliptic K3 surfaces with $\psi$ trivial such that $\zeta^2$ fixes $n$ points and $k$ curves.
When $X$ is $\zeta$-elliptic and $\psi$ is trivial, we see from Table \ref{tab:results} that $X$ and $\zeta$ are defined uniquely by $X^{[2]}$. Hence $\cM_{n,k}^\mathrm{id}=\cM_{n,k}$ and by Theorem \ref{thm:mod} each of them is irreducible and in the closure of $\cM_{0,2}^\mathrm{id}$. 

\item $\cM_{n,k}^{\IZ/2\IZ}$ are the moduli spaces of  $\zeta$-elliptic K3 surfaces with $\psi$ an involution, such that $\zeta^2$ fixes $n$ points and $k$ curves. When $\psi$ is an involution, the polynomial $p_{12}$ is even. Hence, we are again dealing with reduced spaces except for the cases (3/3') and (6/6') in Table \ref{tab:results}, which correspond to different choices of $0$ and $\infty$. Note that in the case (6/6'), there is an extra possibility, topologically equivalent to 6', where $\psi$ is the composite of the $\psi$'s associated to 6 and 6'.\\
Since we can deform even polynomials and preserve evenness, each of these moduli lies in the closure of $\cM_{0,2}^{\IZ/2\IZ}$. Note that each space underlying $\cM_{n,k}^{\IZ/2\IZ}$ is isomorphic to a subvariety of $\cM_{n,k}^\mathrm{id}$. 

\item If we denote by $\cM^{m,n}$ the moduli of $\zeta$-elliptic K3 surfaces wih $\psi$ an involution and having $p_{12}$ with a root of order $2m$ at $0$ and $2n$ at $\infty$. When moving the roots, we can pass from one moduli to another. Their relative position is given in Figure \ref{fig:mod1}. In each case, the fixed locus can be inferred from the fixed locus of the original element in $\cM^{0,0}$.
\end{enumerate}

\end{remark}

\begin{figure}[htp]
\centering
\psfrag{A}{$\cM^{0,0}$}
\psfrag{B}{$\cM^{1,0}$}
\psfrag{C}{$\cM^{2,0}$}
\psfrag{D}{$\cM^{1,1}$}
\psfrag{E}{$\cM^{2,1}$}
\includegraphics[height=4cm]{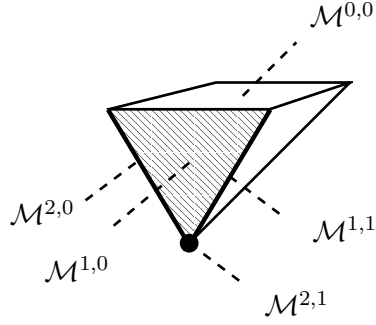}\\
\caption{Relative position of the $\cM^{m,n}$.}\label{fig:mod1}
\end{figure}

%% Fixed Picard Lattice

\section{Fixed Picard lattice.}
\label{sec:pic}

A special case of the above classification consists in analysing only those automorphisms which fix the Picard group. Although this can be recovered from the previous sections, we will try to analyse the case separately to make the analogy with automorphisms of order $2$ and $3$ as studied by \cite{Ni81} and \cite{AS08}. \\
Recall that for a K3 surface, $X$, the cohomology $H^2(X,\IZ)$ is a unimodular lattice of signature $(3,19)$, i.e. it is isomorphic to $U^3\oplus E_8^2$. Also, it decomposes into the Picard lattice, $S_X$, and the transcendental lattice $T_X$: 
$$H^2(X,\IZ)\cong S_X\oplus T_X.$$
Given a lattice $A$, we will write $A^\perp$ for its orthogonal complement and $A^*$ for its dual $\Hom(A,\IZ)$. We say that a lattice is \emph{p-elementary} when $A^*/A=(\IZ/p\IZ)^k$ for some $k\in \IN$. 

\begin{lemma} 
Let $\zeta$ be a primitive non-symplectic automorphism of order $6$ of $X$ which preserves the Picard lattice, then the Picard lattice $S_X=H^2(X,\IZ)^\zeta$ is a unimodular.
\end{lemma}

\begin{proof} 
Fix $p\in \{2,3\}$, and let $\zeta_p=(\zeta^*)^{6/p}$. 
The quotients $S_X^*/S_X$ and $\left(S_X^\perp\right)^* / S_X^\perp=T_X^*/T_X$ are isomorphic. 
Hence,  $p=1+\zeta_p^*+...+(\zeta_p^*)^{p-1}=0$ on $T_X$ and $pT_X^* \subset T_X$. Since $S_X$ is both $2$ and $3$ elementary, it is unimodular.
\end{proof}

\begin{corollary} 
The Picard lattice $S_X$ is isomorphic to $U$, $U\oplus E_8$ or $U\oplus E_8^2$.
\end{corollary}

\begin{proof}
By the Hodge index theorem, $S_X$ is of signature $(1,*)$. By adjunction, the lattice is even. Using the classification of even unimodular lattices, e.g. in \cite{Ser70}, we get the desired result. 
\end{proof}

Since in the three cases $S_X$ decomposes as the direct sum of $U$ with a negative definite lattice, it is easy to see that we fall everytime in the elliptic case. Moreover, using the Lefschetz topological formula or the fact that only the Picard lattice is fixed, one can see that there are no other irreducible fibers except for the given $E_8$ fibers generating part of the Picard lattice. Note that using lemma \ref{lem:tree} one can see that in the case of $\mathrm{rk}\; S_X > 2$ only the section and the rational lines of degree $3$ are fixed. 

Recall the following definition, due to Dolgachev \cite{Dol96}, of mirror pairs for K3 surfaces.

\begin{definition}
The K3 surfaces $(M,W)$ form a mirror pair whenever $S_M^\perp=S_W\oplus U$.
\end{definition}

When applied to the case of unimodular Picard lattices, we see that K3 surfaces form a pair when their Picard groups are respectively $U^i\oplus E_8^j$ and $U^{2-i}\oplus E_8^{2-j}$. I.e. the surfaces with Picard groups $U$ and $U\oplus E_8^2$ are dual to one another while the surfaces with Picard group $U\oplus E_8$ are self-dual. 
This confirms the diagrams obtained for automorphisms of order $2$ and $3$ showing that mirror symmetry is a natural transformation preserving symmetries.

%%% References

\bibliographystyle{halpha}
\bibliography{../sources}

\end{document}

%% file: table.tex
\begin{table}[ht]
\caption{Fixed locus when $X$ is $\zeta$-elliptic.}
\begin{center}
\begin{tabular}{c||ccc|cccc||ccc||cc|ccc}
\hline\\[-0.9em]
\# & g & n & k & $ii$ & $iv$ & $ii^*$ & $iv^*$ & $p_{(3,4)}$ & $p_{(2,5)}$ & $l^{[1]}-1$ & $F_0$ & $F_\infty$ & $p_{(3,4)}$ & $p_{(2,5)}$ & $l^{[1]}$ \\ 
\hline 
\hline\\[-0.7em]
 1 & 5 & 0 & 2 & 12 & 0 & 0 & 0 & 12 & 0 & 0 & $I_0$ & $I_0$ & 6 & 0 & 0 \\ 

 2 & 4 & 1 & 2 & 10 & 1 & 0 & 0 & 10 & 1 & 0 & $I_0$ & $IV$ & 4 & 1 & 0 \\ 

 3 & 3 & 2 & 2 & 8 & 2 & 0 & 0 & 8 & 2 & 0 & $I_0$ & $I_0$ & 6 & 0 & 0 \\ 

 4 & 2 & 3 & 2 & 6 & 3 & 0 & 0 & 6 & 3 & 0 & $I_0$ & $IV$ & 4 & 1 & 0 \\ 
 5 & 3 & 3 & 3 & 8 &  & 0 & 1 & 10 & 1 & 0 & $I_0$ & $IV^*$ & 6 & 3 & 1 \\ 
 6 & 1 & 4 & 2 & 4 & 4 & 0 & 0 & 4 & 4 & 0 & $I_0$ & $I_0$ & 6 & 0 & 0 \\ 
 7 & 2 & 4 & 3 & 6 & 1 & 0 & 1 & 8 & 2 & 0 & $IV$ & $IV^*$ & 4 & 4 & 1 \\ 
 8 & 3 & 4 & 4 & 7 & 0 & 1 & 0 & 10 & 4 & 1 &  &  &  &  &  \\ 
 9 & 0 & 5 & 2 & 2 & 5 & 0 & 0 & 2 & 5 & 0 & $I_0$ & $IV$ & 4 & 1 & 0 \\ 
 10 & 1 & 5 & 3 & 4 & 2 & 0 & 1 & 6 & 3 & 0 & $I_0$ & $IV^*$ & 6 & 3 & 1 \\ 
 11 & 2 & 5 & 4 & 5 & 1 & 1 & 0 & 8 & 5 & 1 &  &  &  &  &  \\ 
 12 & 0 & 6 & 3 & 2 & 3 & 0 & 1 & 4 & 4 & 0 & $IV$ & $IV^*$ & 4 & 4 & 1 \\ 
 13 & 1 & 6 & 4 & 3 & 2 & 1 & 0 & 6 & 6 & 1 &  &  &  &  &  \\ 
 14 & 0 & 7 & 4 & 1 & 3 & 1 & 0 & 4 & 7 & 1 &  &  &  &  &  \\ 
 15 & 1 & 7 & 5 & 3 & 0 & 1 & 1 & 8 & 5 & 1 &  &  &  &  &  \\ 
 16 & 0 & 8 & 5 & 1 & 1 & 1 & 1 & 6 & 6 & 1 &  &  &  &  &  \\ 
 17 & 1 & 8 & 6 & 2 & 0 & 2 & 0 & 8 & 8 & 2 & $I_0$ & $I_0$ & 6 & 0 & 0 \\ 
 18 & 0 & 9 & 6 & 0 & 1 & 2 & 0 & 6 & 9 & 2 & $I_0$ & $I_0$ & 4 & 1 & 0 \\ 
  &  &  &  &  &  &  &  &  &  &  &  &  &  &  &  \\ 
 3' & 3 & 2 & 2 & 8 & 2 & 0 & 0 & 8 & 2 & 0 & $IV$ & $IV$ & 2 & 2 & 0 \\ 
 6' & 1 & 4 & 2 & 4 & 4 & 0 & 0 & 4 & 4 & 0 & $IV$ & $IV$ & 2 & 2 & 0 \\ 
\hline
\end{tabular}
\end{center}
\label{tab:results}
\end{table}

\begin{table}[ht]
\caption{Fixed locus when $X$ is not $\zeta$-elliptic.}
\begin{center}
\begin{tabular}{c|ccc|ccc}
\hline\\[-0.9em]
\#  & g & n & k & $p_{(3,4)}$ & $p_{(2,5)}$ & $l^{[1]}$\\ 
\hline 
\hline\\[-0.7em]
1 &   4 & 0 & 1 & 6 & 0 & 0   \\ 
2 &   3 & 1 & 1 & 4 & 1 & 0   \\
3 &   2 & 2 & 1 & 6 & 0 & 0    \\
4 &   $\emptyset$ & 3 & 0 & 0 & 3 & 0    \\ 
5 &   1 & 3 & 1 & 4 & 1 & 0     \\ 
6 &   0 & 4 & 1 & 2 & 2 & 0     \\ 
  &     &   &   &  &  &  \\ 
3'   & 2 & 2 & 1 & 2 & 2 & 0    \\
\hline
\end{tabular}
\end{center}
\label{tab:results2}
\end{table}

%% file: paper.bbl
\begin{thebibliography}{AST09}

\bibitem[AS08]{AS08}
Michela Artebani and Alessandra Sarti.
\newblock Non-symplectic automorphisms of order 3 on {$K3$} surfaces.
\newblock {\em Math. Ann.}, 342(4):903--921, 2008.

\bibitem[AST09]{AST09}
Michela Artebani, Alessandra Sarti, and Shingo Taki.
\newblock K3 surfaces with non-symplectic automorphisms of prime order.
\newblock 2009, arXiv:0903.3481.

\bibitem[Bor97]{Bor97}
Ciprian Borcea.
\newblock {$K3$} surfaces with involution and mirror pairs of {C}alabi-{Y}au
  manifolds.
\newblock In {\em Mirror symmetry, {II}}, volume~1 of {\em AMS/IP Stud. Adv.
  Math.}, pages 717--743. Amer. Math. Soc., Providence, RI, 1997.

\bibitem[Car57]{Car57}
Henri Cartan.
\newblock Quotient d'un espace analytique par un groupe d'automorphismes.
\newblock In {\em Algebraic geometry and topology}, pages 90--102. Princeton
  University Press, Princeton, N. J., 1957.
\newblock A symposium in honor of S. Lefschetz,.

\bibitem[Dil06]{Dil06}
Jimmy Dillies.
\newblock {\em Automorphisms and Calabi-Yau threefolds}.
\newblock PhD thesis, University of Pennsylvania, Philadelphia, Pennsylvania,
  2006.

\bibitem[Dol96]{Dol96}
I.~V. Dolgachev.
\newblock Mirror symmetry for lattice polarized {$K3$} surfaces.
\newblock {\em J. Math. Sci.}, 81(3):2599--2630, 1996.
\newblock Algebraic geometry, 4.

\bibitem[Fra]{Fran09}
Kristina Frantzen.
\newblock Classification of {K3}-surfaces with involution and maximal
  symplectic symmetry.
\newblock arXiv:0906.1867.

\bibitem[Kon86]{Kon86}
Shigeyuki Kond{\=o}.
\newblock On automorphisms of algebraic {$K3$} surfaces which act trivially on
  {P}icard groups.
\newblock {\em Proc. Japan Acad. Ser. A Math. Sci.}, 62(9):356--359, 1986.

\bibitem[Kon92]{Kon92}
Shigeyuki Kond{\=o}.
\newblock Automorphisms of algebraic {$K3$} surfaces which act trivially on
  {P}icard groups.
\newblock {\em J. Math. Soc. Japan}, 44(1):75--98, 1992.

\bibitem[MO98]{MO98}
Natsumi Machida and Keiji Oguiso.
\newblock On {$K3$} surfaces admitting finite non-symplectic group actions.
\newblock {\em J. Math. Sci. Univ. Tokyo}, 5(2):273--297, 1998.

\bibitem[Nik81]{Ni81}
V.~V. Nikulin.
\newblock Quotient-groups of groups of automorphisms of hyperbolic forms by
  subgroups generated by {$2$}-reflections. {A}lgebro-geometric applications.
\newblock In {\em Current problems in mathematics, {V}ol. 18}, pages 3--114.
  Akad. Nauk SSSR, Vsesoyuz. Inst. Nauchn. i Tekhn. Informatsii, Moscow, 1981.

\bibitem[OZ98]{OZ98}
K.~Oguiso and D.-Q. Zhang.
\newblock {$K3$} surfaces with order five automorphisms.
\newblock {\em J. Math. Kyoto Univ.}, 38(3):419--438, 1998.

\bibitem[OZ00]{OZ00}
Keiji Oguiso and De-Qi Zhang.
\newblock On {V}orontsov's theorem on {$K3$} surfaces with non-symplectic group
  actions.
\newblock {\em Proc. Amer. Math. Soc.}, 128(6):1571--1580, 2000.

\bibitem[Ser70]{Ser70}
Jean-Pierre Serre.
\newblock {\em Cours d'arithm\'etique}, volume~2 of {\em Collection SUP: ``Le
  Math\'ematicien''}.
\newblock Presses Universitaires de France, Paris, 1970.

\bibitem[TT75]{TT75}
Domingo Toledo and Yue Lin~L. Tong.
\newblock The holomorphic {L}efschetz formula.
\newblock {\em Bull. Amer. Math. Soc.}, 81(6):1133--1135, 1975.

\bibitem[Voi93]{Voi93}
Claire Voisin.
\newblock Miroirs et involutions sur les surfaces {$K3$}.
\newblock {\em Ast\'erisque}, (218):273--323, 1993.
\newblock Journ{\'e}es de G{\'e}om{\'e}trie Alg{\'e}brique d'Orsay (Orsay,
  1992).

\bibitem[Vor83]{Vo83}
S.~P. Vorontsov.
\newblock Automorphisms of even lattices arising in connection with
  automorphisms of algebraic {$K3$}-surfaces.
\newblock {\em Vestnik Moskov. Univ. Ser. I Mat. Mekh.}, (2):19--21, 1983.

\bibitem[Xia96]{Xi96}
Gang Xiao.
\newblock Galois covers between {$K3$} surfaces.
\newblock {\em Ann. Inst. Fourier (Grenoble)}, 46(1):73--88, 1996.

\bibitem[Zha07]{Zh07}
De-Qi Zhang.
\newblock Automorphisms of {$K3$} surfaces.
\newblock In {\em Proceedings of the {I}nternational {C}onference on {C}omplex
  {G}eometry and {R}elated {F}ields}, volume~39 of {\em AMS/IP Stud. Adv.
  Math.}, pages 379--392, Providence, RI, 2007. Amer. Math. Soc.

\end{thebibliography}
